\documentclass{amsart}
\usepackage[all]{xy}
\usepackage{verbatim}
\usepackage{color}
\usepackage{amsthm}
\usepackage{amssymb}
\usepackage[colorlinks=true]{hyperref}
\usepackage{graphicx}

\numberwithin{equation}{section}

\newtheorem{theorem}[equation]{Theorem}
\newtheorem*{theorem*}{Theorem}

\newtheorem*{conjecture*}{Mamma Conjecture}
\newtheorem*{conjecture1*}{Mamma Conjecture (revisited)}
\newtheorem{proposition}[equation]{Proposition}

\newtheorem*{corollary*}{Corollary}

\theoremstyle{remark}
\newtheorem{definition}[equation]{Definition}

\newtheorem{example}[equation]{Example}

\theoremstyle{remark}
\newtheorem{remark}[equation]{Remark}

\newcommand{\cA}{{\mathcal A}}

\newcommand{\cC}{{\mathcal C}}
\newcommand{\cD}{{\mathcal D}}

\newcommand{\cH}{{\mathcal H}}

\newcommand{\cS}{{\mathcal S}}
\newcommand{\cT}{{\mathcal T}}

\newcommand{\bbA}{\mathbb{A}}

\newcommand{\bbF}{\mathbb{F}}

\newcommand{\bbR}{\mathbb{R}}
\newcommand{\bbS}{\mathbb{S}}

\newcommand{\bbQ}{\mathbb{Q}}
\newcommand{\bbZ}{\mathbb{Z}}

\DeclareMathOperator{\Id}{Id}
\DeclareMathOperator{\id}{id}

\newcommand{\Hom}{\mathrm{Hom}}

\newcommand{\op}{\mathrm{op}}

\newcommand{\too}{\longrightarrow}

\newcommand{\ie}{\textsl{i.e.}\ }
\newcommand{\eg}{\textsl{e.g.}}

\let\oldmarginpar\marginpar
\def\marginpar#1{\oldmarginpar{\tiny #1}}

\begin{document}

\title[Picard groups, weight structures, and (NC) mixed motives]{Picard groups, weight structures, and (noncommutative) mixed motives}
\author{Mikhail Bondarko and Gon{\c c}alo~Tabuada}

\address{Mikhail Bondarko, Russia, 198504, St. Petersburg, Universitetsky pr. 28, St. Petersburg State University, the Faculty of Higher Algebra and Number Theory }
\email{m.bondarko@spbu.ru}

\address{Gon{\c c}alo Tabuada, Department of Mathematics, MIT, Cambridge, MA 02139, USA}
\email{tabuada@math.mit.edu}
\urladdr{http://math.mit.edu/~tabuada}
\thanks{M.~Bondarko was supported by RFBR
(grants no.~14-01-00393-a and 15-01-03034-a), by  Dmitry Zimin's Foundation ``Dynasty", and by the Scientific schools grant no.~3856.2014.1. G.~Tabuada was partially supported by a NSF CAREER Award}


\date{\today}


\abstract{We develop a general theory which enables the computation of the Picard group of a symmetric monoidal triangulated category, equipped with a weight structure, in terms of the Picard group of the associated heart. As an application, we compute the Picard group of several categories of motivic nature  -- mixed Artin motives, mixed Artin-Tate motives, motivic spectra, noncommutative mixed Artin motives, noncommutative mixed motives of central simple algebras, noncommutative mixed motives of separable algebras -- as well as the Picard group of the derived categories of symmetric ring spectra.}
}

\maketitle
\vskip-\baselineskip
\vskip-\baselineskip

\section{Introduction and statement of results}\label{sec:intro}
The computation of the Picard group $\mathrm{Pic}(\cT)$ of a symmetric monoidal (triangulated) category $\cT$ is, in general, a very difficult task. The goal of this article is to explain how the theory of weight structures allows us to greatly simplify this task.

Let $(\cT, \otimes, {\bf 1})$ be a symmetric monoidal triangulated category equipped with a weight structure $w=(\cT^{w\geq 0}, \cT^{w\leq 0})$; consult \S\ref{sec:weight} for details. Assume that the symmetric monoidal structure $-\otimes-$ (as well as the $\otimes$-unit ${\bf 1}$) restricts to the heart $\cH:=\cT^{w \geq 0} \cap \cT^{w \leq 0}$ of the weight structure. We say that the category $\cT$ has the {\em $w$-Picard property} if the group homomorphism $\mathrm{Pic}(\cH) \times \bbZ \to \mathrm{Pic}(\cT), (a, n) \mapsto a[n]$,~is invertible. Our first main result provides sufficient conditions for this property~to~hold:
\begin{theorem}\label{thm:main1}
Assume that the weight structure $w$ on $\cT$ is {\em bounded}, \ie $\cT=\cup_{n \in \bbZ} \cT^{w\geq 0}[-n]=\cup_{n \in \bbZ} \cT^{w \leq 0}[-n]$, and that the Karoubization of the heart $\cH$ is semi-simple and {\em local} in the sense that if $a\otimes b=0$ then $a=0$ or $b=0$. Under these assumptions, the category $\cT$ has the $w$-Picard property.
\end{theorem}
As explained in \cite[\S4.3]{Bondarko-Weight}, every bounded weight structure is uniquely determined by its heart. Concretely, given any additive subcategory $\cH' \subset \cT$ which generates $\cT$ and for which we have $\Hom_{\cH'}(a,b[n])=0$ for every $n>0$ and $a,b \in \cH'$, there exists a unique bounded weight structure on $\cT$ with heart the Karoubi-closure of $\cH'$ in $\cT$. Roughly speaking, the construction of a bounded weight structure on a triangulated category amounts simply to the choice of an additive subcategory with trivial positive Ext-groups.

Our second main result formalizes the conceptual idea that the $w$-Picard property satisfies a ``global-to-local'' descent principle:
\begin{theorem}\label{thm:main2}
Assume the following:
\begin{itemize}
\item[(A1)] The heart $\cH$ of the weight structure $w$ is essentially small and $R$-linear for some commutative indecomposable Noetherian ring $R$. Moreover, $\Hom_\cH(a,b)$ is a finitely generated flat $R$-module for any two objects $a, b \in \cH$;
\item[(A2)] For every residue field $\kappa(\mathfrak{p})$, with $\mathfrak{p} \in \mathrm{Spec}(R)$, there exists a symmetric monoidal triangulated category $(\cT_{\kappa(\mathfrak{p})}, \otimes, {\bf 1})$ equipped with a weight structure $w_{\kappa(\mathfrak{p})}$ and with a weight-exact symmetric monoidal functor $\iota_{\kappa(\mathfrak{p})}\colon \cT \to \cT_{\kappa(\mathfrak{p})}$. Moreover, the functor $\iota_{\kappa(\mathfrak{p})}$ induces an equivalence of categories between the Karoubization of $\cH\otimes_R \kappa(\mathfrak{p})$ and $\cH_{\kappa(\mathfrak{p})}$. 
\end{itemize}
Under assumptions (A1)-(A2), if the categories $\cT_{\kappa(\mathfrak{p})}$ have the $w_{\kappa(\mathfrak{p})}$-Picard property, then the category $\cT$ has the $w$-Picard property.
\end{theorem}
\begin{remark}\label{rk:maximal}
\begin{itemize}
\item[(i)] At assumption (A1) we can consider more generally the case where $R$ is decomposable; consult Remark \ref{rcoeff}(i).
\item[(ii)] As it will become clear from the proof of Theorem \ref{thm:main2}, at assumption (A2) it suffices to consider the residue fields $\kappa(\mathfrak{m})$ associated to the maximal and minimal prime ideals of $R$; consult also Remark \ref{rcoeff}(ii).
\end{itemize}
\end{remark}
Due to their generality and simplicity, we believe that Theorems \ref{thm:main1}-\ref{thm:main2} will soon be part of the toolkit of every mathematician interested in Picard groups of triangulated categories. In the next section, we illustrate the usefulness of these results by computing the Picard group of several important categories of motivic nature; consult also \S\ref{sub:topological} for a topological application.
\section{Applications}\label{sec:applications}
Let $k$ be a base field, which we assume perfect, and $R$ a commutative ring of coefficients, which we assume indecomposable and Noetherian. Voevodsky's category of geometric mixed motives $\mathrm{DM}_{\mathrm{gm}}(k; R)$ (see \cite{Princeton, MVW}), Morel-Voevodsky's stable $\bbA^1$-homotopy category $\mathrm{SH}(k)$ (see \cite{Morel1,Morel2,ICM}), and Kontsevich's category of noncommutative mixed motives $\mathrm{KMM}(k; R)$ (see \cite{IAS,Miami,finMot,book}) play nowadays a central role in the motivic realm. A major challenge, which seems completely out of reach at the present time, is the computation of the Picard group of these symmetric monoidal triangulated categories\footnote{Consult Bachmann \cite{bach}, resp. Hu \cite{PoHu}, for the construction of $
\otimes$-invertible objects in the motivic category $\mathrm{DM}_{\mathrm{gm}}(k; \bbZ/2\bbZ)$, resp. $\mathrm{SH}(k)$, associated to quadrics.}. In what follows, making use of Theorems \ref{thm:main1}-\ref{thm:main2}, we achieve this goal in the case~of~certain~important~subcategories. 
\subsection{Mixed Artin motives}
The category of {\em mixed Artin motives $\mathrm{DMA}(k;R)$} is defined as the thick triangulated subcategory of $\mathrm{DM}_{\mathrm{gm}}(k; R)$ generated by the motives $M(X)_R$ of zero-dimensional smooth $k$-schemes $X$. The smallest additive, Karoubian, full subcategory of $\mathrm{DMA}(k;R)$ containing the objects $M(X)_R$ identifies with the (classical) category of Artin motives $\mathrm{AM}(k;R)$.
\begin{theorem}\label{thm:computation-1}
When the degrees of the finite separable field extensions of $k$ are invertible in $R$, we have $\mathrm{Pic}(\mathrm{DMA}(k;R)) \simeq \mathrm{Pic}(\mathrm{AM}(k;R))\times \bbZ$.
\end{theorem}
\begin{example}\label{ex:cases}
Theorem \ref{thm:computation-1} holds, in particular, in the following cases:
\begin{itemize}
\item[(i)] The field $k$ is arbitrary and $R$ is a $\bbQ$-algebra;
\item[(ii)] The field $k$ is formally real (\eg\ $k=\bbR$) and $1/2 \in R$;
\item[(iii)] Let $p$ be a (fixed) prime number, $l$ a perfect field, and $H$ a Sylow pro-$p$-subgroup of $\mathrm{Gal}(\overline{l}/l)$. Theorem \ref{thm:computation-1} holds also with $k:=\overline{l}^H$ and $1/p \in R$.
\end{itemize}
\end{example}
Whenever $R$ is a field, the $R$-linearized Galois-Grothendieck correspondence induces a $\otimes$-equivalence between $\mathrm{AM}(k;R)$ and the category $\mathrm{Rep}_R(\Gamma)$ of continuous finite dimensional $R$-linear representations of the absolute Galois group $\Gamma:=\mathrm{Gal}(\overline{k}/k)$. Since the $\otimes$-invertible objects of $\mathrm{Rep}_R(\Gamma)$ are the $1$-dimensional $\Gamma$-representations, the Picard group $\mathrm{Pic}(\mathrm{AM}(k;R))\simeq \mathrm{Pic}(\mathrm{Rep}_R(\Gamma))$ identifies then with the group of continuous characters from $\Gamma^{\mathrm{ab}}$ to $R^\times$. In the particular case where $k=\bbQ$, the profinite group $\Gamma^{\mathrm{ab}}$ identifies with $\widehat{\bbZ}^\times$. Consequently, all the elements of $\mathrm{Rep}_R(\Gamma)$ can be represented by Dirichlet characters. In the particular case where $\mathrm{char}(k)\neq 2$ and $R=\bbQ$, we have the following computation\footnote{A similar computation holds in characteristic $2$ with $k^\times/(k^\times)^2$ replaced by $k/\{\lambda + \lambda^2\,|\, \lambda \in k\}$.} (due to Peter \cite[Pages~340-341]{Peter})
\begin{eqnarray*}
k^\times/(k^\times)^2\stackrel{\simeq}{\too} \mathrm{Pic}(\mathrm{Rep}_\bbQ(\Gamma)) && \lambda \mapsto (\Gamma \twoheadrightarrow \mathrm{Gal}(k(\sqrt{\lambda})/k) \stackrel{\sigma \mapsto -1}{\too} \bbQ^\times)\,,
\end{eqnarray*}
where $\sigma$ stands for the generator of the Galois group $\mathrm{Gal}(k(\sqrt{\lambda})/k)\simeq \bbZ/2\bbZ$.

Let $\cA(k;R)$ be an additive, Karoubian, symmetric monoidal, full subcategory of $\mathrm{AM}(k;R)$, and $\mathrm{D}\cA(k;R)$ the thick triangulated subcategory of $\mathrm{DMA}(k;R)$ generated by the motives associated to the objects of $\cA(k;R)$. Under these notations, Theorem \ref{thm:computation-1} admits the following generalization:
\begin{theorem}\label{thm:computation-2}
Assume that there exist separable field extensions~$l_i/k$~such~that:
\begin{itemize}
\item[(B1)] Every object of the category $\cA(k;R)$ is isomorphic to a retract of a finite direct sum of the motives associated to the field extensions $l_i/k$;
\item[(B2)] The degrees of the finite field extensions $l_i/k$ are invertible in $R$.
\end{itemize}
Under assumptions (B1)-(B2), we have $\mathrm{Pic}(\mathrm{D}\cA(k;R)) \simeq \mathrm{Pic}(\cA(k;R)) \times \bbZ$.
\end{theorem}
\begin{example}[Mixed Dirichlet motives]\label{ex:MDM}
Let $R$ be a field. Following Wildeshaus \cite[Def.~3.4]{Jorg}, a {\em Dirichlet motive} is an Artin motive for which the corresponding $\Gamma$-representation factors through an abelian (finite) quotient. Let $\cA(k;R)$ be the category of Dirichlet motives. In this case, the associated symmetric monoidal triangulated category $\mathrm{D}\cA(k;R)$ is called the category of {\em mixed Dirichlet motives}. Since the $\otimes$-invertible objects of $\mathrm{Rep}_R(\Gamma)$ are the $1$-dimensional representations, and all these representations factor through an abelian (finite) quotient, the inclusion of categories $\cA(k;R) \subset \mathrm{AM}(k;R)$ yields an isomorphism $\mathrm{Pic}(\cA(k;R)) \simeq \mathrm{Pic}(\mathrm{AM}(k;R))$. Making use of Theorem \ref{thm:computation-2}, we conclude that $\mathrm{Pic}(\mathrm{D}\cA(k;R))\simeq \mathrm{Pic}(\mathrm{AM}(k;R))\times \bbZ$. Intuitively speaking, the difference between (mixed) Dirichlet motives and (mixed) Artin motives is not detected by the Picard group. 
\end{example}
%
%
%
%
%
\subsection{Mixed Artin-Tate motives}
The category $\mathrm{DMAT}(k;R)$ of {\em mixed Artin-Tate motives} is defined as the thick triangulated subcategory of $\mathrm{DM}_{\mathrm{gm}}(k;R)$ generated by the motives $M(X)_R$ of zero-dimensional smooth $k$-schemes $X$ and by the Tate motives $R(m), m \in \bbZ$.
\begin{theorem}\label{thm:computation-3}
When the degrees of the finite separable field extensions of $k$ are invertible in $R$, we have $\mathrm{DMAT}(k;R)\simeq \mathrm{Pic}(\mathrm{AM}(k;R))\times \bbZ \times \bbZ$.
\end{theorem}
Let $\cA(k;R)$ be an additive, Karoubian, symmetric monoidal, full subcategory of $\mathrm{AM}(k;R)$, and $\mathrm{D}\cA\mathrm{T}(k;R)$ the thick triangulated subcategory of $\mathrm{DMAT}(k;R)$ generated by the motives associated to the objects of $\cA(k;R)$ and by the Tate motives $R(m), m \in \bbZ$. Theorem \ref{thm:computation-3} admits the following generalization:
\begin{theorem}\label{thm:computation-4}
Assume that there exist finite separable field extensions $l_i/k$ as in Theorem \ref{thm:computation-2}. Under these assumptions, $\mathrm{Pic}(\mathrm{D}\cA\mathrm{T}(k;R))\simeq \mathrm{Pic}(\cA(k;R)) \times \bbZ \times \bbZ$.
\end{theorem}
\begin{example}[Mixed Tate motives]\label{ex:MixedTate}
Let $\cA(k;R)$ be the smallest additive, Karoubian, full subcategory of $\mathrm{AM}(k;R)$ containing the $\otimes$-unit. In this case, the associated symmetric monoidal triangulated category $\mathrm{D}\cA \mathrm{T}(k;R)$ is called the category of {\em mixed Tate motives}. Since $\cA(k;R)$ identifies with the category of finitely generated projective $R$-modules\footnote{Recall that the Picard group $\mathrm{Pic}(R)$ of a Dedekind domain $R$ is its ideal class group $C(R)$.}, we conclude from Theorem \ref{thm:computation-4} that $\mathrm{Pic}(\mathrm{D}\cA\mathrm{T}(k;R))\simeq \mathrm{Pic}(R) \times \bbZ \times \bbZ$. Note that we are not imposing the invertibility of any integer~in~$R$.

\end{example}
\begin{example}[Mixed Dirichlet-Tate motives]
Let $\cA(k;R)$ be the category of Dirichlet motives. In this case, the associated symmetric monoidal triangulated category $\mathrm{D}\cA\mathrm{T}(k;R)$ is called the category of {\em mixed Dirichlet-Tate motives}. Since the Picard group of $\cA(k;R)$ is isomorphic to the Picard group of $\mathrm{AM}(k;R)$, we conclude from Theorem \ref{thm:computation-4} that $\mathrm{Pic}(\mathrm{D}\cA\mathrm{T}(k;R))\simeq \mathrm{Pic}(\mathrm{AM}(k; R))\times \bbZ \times \bbZ$.
\end{example}
\subsection{Motivic spectra}
The {\em bootstrap category $\mathrm{Boot}(k)$} is defined as the thick triangulated subcategory of $\mathrm{SH}(k)$ generated by the $\otimes$-unit $\Sigma^\infty(\mathrm{Spec}(k)_+)$. The former category contains a lot of information. For example, as proved by Levine in \cite[Thm.~1]{Levine}, whenever $k$ is algebraically closed and of characteristic zero, $\mathrm{Boot}(k)$ identifies with the homotopy category of finite spectra $\cS\cH_c$. In particular, we have non-trivial negative Ext-groups
\begin{eqnarray}\label{eq:non-trivial-verylast}
\Hom_{\mathrm{Boot}(k)}(\Sigma^\infty(\mathrm{Spec}(k)_+), \Sigma^\infty(\mathrm{Spec}(k)_+)[-n])\simeq \pi_n(\bbS) && n>0\,,
\end{eqnarray}
where $\bbS$ stands for the sphere spectrum. Moreover, as proved by Morel in \cite[Thm.~6.2.2]{Morel3}, whenever $k$ is of characteristic $\neq 2$, we have a ring isomorphism 
\begin{equation}\label{eq:ring-iso}
\mathrm{End}_{\mathrm{Boot}(k)}(\Sigma^\infty(\mathrm{Spec}(k)_+))\simeq GW(k)\,,
\end{equation}
where $GW(k)$ stands for the Grothendieck-Witt ring of $k$. 
\begin{theorem}\label{thm:computation-last}
Assume that $\mathrm{char}(k)\neq 2$ and that $GW(k)$ is Noetherian. Under these assumptions, we have $\mathrm{Pic}(\mathrm{Boot}(k))\simeq \mathrm{Pic}(GW(k))\times \bbZ$.
\end{theorem}
\begin{remark}
The ring $GW(k)$ is Noetherian if and only if $k^\times/(k^\times)^2$ is finite.
\end{remark}
\begin{example}
Theorem \ref{thm:computation-last} holds, in particular, in the following cases:
\begin{itemize}
\item[(i)] The field $k$ is quadratically closed (\eg\ $k$ is algebraically closed or the field of constructible numbers). In this case, we have $GW(k)\simeq \bbZ$;
\item[(ii)] The field $k$ is the field of real numbers $\bbR$. In this case, we have $GW(\bbR)\simeq \bbZ[C_2]$, where $C_2$ stands for the cyclic group of order $2$;
\item[(iii)] The field $k$ is the finite field $\bbF_q$ with $q$ odd. In this case, $k^\times/(k^\times)^2=C_2$.
\end{itemize}
\end{example}
Theorem \ref{thm:computation-last} shows that, whenever $k=\overline{k}$ and $\mathrm{char}(k)=0$, none of the motivic spectra which are built using the non-trivial Ext-groups \eqref{eq:non-trivial-verylast} is $\otimes$-invertible!
\subsection{Noncommutative mixed Artin motives}
The category of {\em noncommutative mixed Artin motives $\mathrm{NMAM}(k; R)$} is defined as the thick triangulated subcategory of $\mathrm{KMM}(k; R)$ generated by the noncommutative motives $U(l)_R$ of finite separable field extensions $l/k$. The smallest additive, Karoubian, full subcategory of $\mathrm{KMM}(k;R)$ containing the objects $U(l)_R$ identifies with $\mathrm{AM}(k;R)$. 

The category of noncommutative mixed Artin motives is in general much richer than the category of mixed Artin motives. For example, whenever $R$ is a $\bbQ$-algebra, $\mathrm{DMA}(k;R)$ identifies with the category $\mathrm{Gr}_\bbZ \mathrm{AM}(k;R)$ of $\bbZ$-graded objects in $\mathrm{AM}(k;R)$; see \cite[Page~217]{Voevodsky}. This implies that $\mathrm{DMA}(k;R)$ has trivial higher Ext-groups. On the other hand, given any two finite separable field extensions $l_1/k$ and $l_2/k$, we have non-trivial negative Ext-groups (see \cite[\S4]{Hopf})
\begin{eqnarray}\label{eq:computation-1}
\Hom_{\mathrm{NMAM}(k;R)}(U(l_1)_R, U(l_2)_R[-n])\simeq K_n(l_1 \otimes_k l_2)_R && n >0\,,
\end{eqnarray}
where $K_n(l_1 \otimes_k l_2)$ stands for the $n^{\mathrm{th}}$ algebraic $K$-theory group of $l_1 \otimes_k l_2$. Roughly speaking, $\mathrm{NMAM}(k;R)$ contains not only $\mathrm{AM}(k;R)$ but also all the higher algebraic $K$-theory groups of finite separable field extensions. For example, given a number field $\bbF$, we have the following computation (due to Borel \cite[\S12]{Borel})
\begin{eqnarray*}
\Hom_{\mathrm{NMAM}(\bbQ; \bbQ)}(U(\bbQ)_\bbQ,U(\bbF)_\bbQ[-n])\simeq \left\{ \begin{array}{ll}
         \bbQ^{r_2}& n\equiv 3\,\,\, (\mathrm{mod}\,4) \\
         \bbQ^{r_1+r_2} & n\equiv 1\,\,\, (\mathrm{mod}\,4) \\
         0 & \mathrm{otherwise}\end{array} \right. && n \geq 2\,,
\end{eqnarray*}
where $r_1$ (resp. $r_2$) stands for the number of real (resp. complex) embeddings of $\bbF$.
\begin{theorem}\label{thm:computation-5}
When the degrees of the finite separable field extensions of $k$ are invertible in $R$, we have $\mathrm{Pic}(\mathrm{NMAM}(k;R))\simeq \mathrm{Pic}(\mathrm{AM}(k;R)) \times \bbZ$.
\end{theorem}
\begin{example}
Theorem \ref{thm:computation-5} holds in the cases (i)-(iii) of Example \ref{ex:cases}. 
\end{example}
Theorem \ref{thm:computation-5} shows that although the category $\mathrm{NMAM}(k;R)$ is much richer than $\mathrm{DMA}(k;R)$, this richness is not detected by the Picard group. In particular, none of the noncommutative mixed motives which are built using the non-trivial negative Ext-groups \eqref{eq:computation-1} is $\otimes$-invertible! 
%

Let $\cA(k;R)$ be an additive, Karoubian, symmetric monoidal, full subcategory of $\mathrm{AM}(k;R)$, and $\mathrm{NM}\cA(k;R)$ the thick triangulated subcategory of $\mathrm{NMAM}(k;R)$ generated by the noncommutative motives associated to the objects of $\cA(k;R)$. Theorem \ref{thm:computation-5} admits the following generalization:
\begin{theorem}\label{thm:computation-6}
Assume that there exist finite separable field extensions $l_i/k$ as in Theorem \ref{thm:computation-2}. Under these assumptions, $\mathrm{Pic}(\mathrm{NM}\cA(k;R))\simeq \mathrm{Pic}(\cA(k;R)) \times \bbZ$.
\end{theorem}
\begin{example}[Noncommutative mixed Dirichlet motives]
Let $\cA(k;R)$ be the category of Dirichlet motives. In this case, the associated symmetric monoidal triangulated category $\mathrm{NM}\cA(k;R)$ is called the category of {\em noncommutative mixed Dirichlet motives}. Since the Picard group of $\cA(k;R)$ is isomorphic to $\mathrm{Pic}(\mathrm{AM}(k;R))$, we conclude from Theorem \ref{thm:computation-5} that $\mathrm{Pic}(\mathrm{NM}\cA(k;R))\simeq \mathrm{Pic}(\mathrm{AM}(k;R))\times \bbZ$. Roughly speaking, the difference between mixed Dirichlet motives (see Example \ref{ex:MDM}) and noncommutative mixed Dirichlet motives is not detected~by~the~Picard~group. 
%
%
%
\end{example}
\begin{example}[Bootstrap category]\label{ex:bootstrap} 
Let $\cA(k;R)$ be the smallest additive, Karoubian, full subcategory of $\mathrm{AM}(k; R)$ containing the $\otimes$-unit. In this case, the associated symmetric monoidal triangulated category $\mathrm{NM}\cA(k;R)$ is called the {\em bootstrap category}. Since $\cA(k;R)$ identifies with the category of finitely generated projective $R$-modules, we conclude from Theorem \ref{thm:computation-6} that $\mathrm{Pic}(\mathrm{NM}\cA(k;R))\simeq \mathrm{Pic}(R)\times \bbZ$. Similarly to Example \ref{ex:MixedTate}, we are not imposing the invertibility of any integer in $R$.

\end{example}

%

\subsection{Noncommutative mixed motives of central simple algebras}
Let us denote by $\mathrm{NMCSA}(k;R)$ the thick triangulated subcategory of $\mathrm{KMM}(k;R)$ generated by the noncommutative motives $U(A)_R$ of central simple $k$-algebras $A$. In the same vein, let $\mathrm{CSA}(k;R)$ be the smallest additive, Karoubian, full subcategory of $\mathrm{NMCSA}(k;R)$ containing the objects $U(A)_R$. As proved in \cite[Thm.~9.1]{Homogeneous}, given any two central simple $k$-algebras $A$ and $B$, we have the equivalence \begin{equation}\label{eq:equivalence}
U(A)_\bbZ \simeq U(B)_\bbZ \Leftrightarrow [A]=[B] \in \mathrm{Br}(k)\,,
\end{equation}
where $\mathrm{Br}(k)$ stands for the Brauer group of $k$. Intuitively speaking, \eqref{eq:equivalence} shows that the noncommutative motive $U(A)_\bbZ$ and the Brauer class $[A]$ contain exactly the same information. We have moreover non-trivial negative Ext-groups:
\begin{eqnarray}\label{eq:ext}
\Hom_{\mathrm{NMCSA}(k;\bbZ)}(U(A)_\bbZ, U(B)_\bbZ[-n])\simeq K_n(A^\op \otimes_k B) && n >0\,.
\end{eqnarray}
For example, when $n=1$ the right-hand side of \eqref{eq:ext} equals $D^\times/[D^\times, D^\times]$, where $D$ stands for the unique division $k$-algebra provided by the Wedderburn theorem $A^\op \otimes_k B \simeq M_{r\times r}(D)$. Roughly speaking, the triangulated category $\mathrm{NMCSA}(k;\bbZ)$ contains information not only about the Brauer group but also about all the higher algebraic $K$-theory of central simple algebras.
\begin{theorem}\label{thm:computation-7}
The following holds:
\begin{itemize}
\item[(i)] We have an isomorphism $\mathrm{Pic}(\mathrm{NMCSA}(k;R))\simeq \mathrm{Pic}(\mathrm{CSA}(k;R))\times \bbZ$;
\item[(ii)] We have an isomorphism $\mathrm{Pic}(\mathrm{CSA}(k;\bbZ))\simeq \mathrm{Br}(k)$.
\end{itemize}
\end{theorem}
\begin{remark}
As it will become clear from the proof of Theorem \ref{thm:computation-7}, we have also $\mathrm{Pic}(\mathrm{CSA}(k;R))\simeq \mathrm{Pic}(R)$ whenever $R$ is a $\bbQ$-algebra and $\mathrm{Pic}(\mathrm{CSA}(k;R))\simeq \mathrm{Br}(k)\{p\}$ whenever $R$ is a field of positive characteristic $p>0$.
\end{remark}
Theorem \ref{thm:computation-7} shows, in particular, that none of the noncommutative mixed motives which are built using the non-trivial negative Ext-groups \eqref{eq:ext} is $\otimes$-invertible!
\subsection{Noncommutative mixed motives of separable algebras}
Recall that a commutative $k$-algebra $A$ is separable if and only if it is isomorphic to a finite product of finite separable field extensions. Therefore, the (classical) category of Artin motives $\mathrm{AM}(k;R)$ can be alternatively defined as the smallest Karoubian full subcategory of $\mathrm{DM}_{\mathrm{gm}}(k;R)$, resp. $\mathrm{KMM}(k;R)$, containing the objects $M(\mathrm{Spec}(A))_R$, resp. $U(A)_R$, with $A$ a commutative separable $k$-algebra. In the setting of noncommutative motives, we can work more generally with separable algebras!

Let us denote by $\mathrm{NMSep}(k;R)$ the thick triangulated subcategory of $\mathrm{KMM}(k;R)$ generated by the noncommutative motives $U(A)_R$ of separable $k$-algebras $A$. In the same vein, let $\mathrm{Sep}(k;R)$ be the smallest additive, Karoubian, full subcategory of $\mathrm{NMSep}(k;R)$ containing the objects $U(A)_R$.
\begin{theorem}\label{thm:computation-8}
When the degrees of the finite separable field extensions of $k$ are invertible in $R$, we have $\mathrm{Pic}(\mathrm{NMSep}(k;R))\simeq \mathrm{Pic}(\mathrm{Sep}(k;R))\times \bbZ$.
\end{theorem}
\begin{example}
Theorem \ref{thm:computation-8} holds in the cases (i)-(iii) of Example \ref{ex:cases}. 
\end{example}
As proved in \cite[Cor.~2.13]{separable}, the assignment $A \mapsto Z(A)$, where $Z(A)$ stands for the center of $A$, gives rise to an additive symmetric monoidal functor $Z(-)$ from $\mathrm{Sep}(k;R)$ to $\mathrm{AM}(k;R)$. Given an additive, Karoubian, symmetric monoidal, full subcategory $\cA(k;R) \subset \mathrm{AM}(k;R)$, we can then consider the associated fiber product
$$
\xymatrix{
\mathrm{S}\cA(k;R) \ar[d] \ar[r] \ar@{}|{\ulcorner}[dr]& \mathrm{Sep}(k;R) \ar[d]^-{Z(-)} \\
\cA(k;R) \ar[r] & \mathrm{AM}(k;R)\,.
}
$$
Let us denote by $\mathrm{NMS}\cA(k;R)$ the thick triangulated subcategory of $\mathrm{NMSep}(k;R)$ generated by the noncommutative motives associated to the objects of $\mathrm{S}\cA(k;R)$. Theorem \ref{thm:computation-8} admits the following generalization:
\begin{theorem}\label{thm:computation-9}
Assume that there exist finite separable field extensions $l_i/k$ as in Theorem \ref{thm:computation-2}. Under these assumptions, $\mathrm{Pic}(\mathrm{NMS}\cA(k;R))\simeq \mathrm{Pic}(\mathrm{S}\cA(k;R)) \times \bbZ$.
\end{theorem}
\begin{remark}
\begin{itemize}
\item[(i)] The functor $Z(-)$ admits a section given by the inclusion of categories $\mathrm{AM}(k;R) \subset \mathrm{Sep}(k;R)$. This implies that $\mathrm{Pic}(\mathrm{S}\cA(k;R))$ decomposes into the product of $\mathrm{Pic}(\cA(k;R))$ with the kernel of $\mathrm{Pic}(\mathrm{S}\cA(k;R))\to \mathrm{Pic}(\cA;R)$;
\item[(ii)] When $\cA(k;R)$ is the smallest additive, Karoubian, full subcategory of $\mathrm{AM}(k;R)$ containing the $\otimes$-unit, $\mathrm{S}\cA(k;R)$ identifies with $\mathrm{CSA}(k;R)$, and consequently Theorem \ref{thm:computation-9} reduces to Theorem \ref{thm:computation-7}(i);
\item[(iii)] Whenever $R$ is a $\bbQ$-algebra, the functor $Z(-)$ is an equivalence of categories; see Remark \ref{rk:last}. Consequently, $\mathrm{S}\cA(k;R)$ identifies with $\cA(k;R)$, and Theorems \ref{thm:computation-8} and \ref{thm:computation-9} reduce to Theorems \ref{thm:computation-5} and \ref{thm:computation-6}, respectively.
\end{itemize}
\end{remark}
\subsection{A topological application}\label{sub:topological}
Let $E$ be a commutative symmetric ring spectrum and $\cD_c(E)$ the associated derived category of compact $E$-modules; see \cite{HSS,Spectra1}.
\begin{theorem}\label{thm:computation-10}
Assume that the ring spectrum $E$ is {\em connective}, \ie $\pi_n(E)=0$ for every $n<0$, and that $\pi_0(E)$ is an indecomposable Noetherian ring. Under these assumptions, we have $\mathrm{Pic}(\cD_c(E))\simeq \mathrm{Pic}(\pi_0(E))\times \bbZ$.
\end{theorem}
\begin{example}[Finite spectra]
Let $E$ be the sphere spectrum $\bbS$. In this case, the category $\cD_c(\bbS)$ is equivalent to the homotopy category of finite spectra $\cS\cH_c$ and $\pi_0(\bbS)\simeq \bbZ$. Consequently, we obtain $\mathrm{Pic}(\cS\cH_c)\simeq \bbZ$. This computation was originally established by Hopkins-Mahowald-Sadofsky in \cite{HMS} using different tools. Note that this computation may be understood as a particular case of Theorem~\ref{thm:computation-last}.
\end{example}
\begin{example}[Ordinary rings]
Let $E$ be the Eilenberg-MacLane spectrum $HR$ of a commutative indecomposable Noetherian ring $R$. In this case, $\cD_c(HR)\simeq \cD_c(R)$ and $\pi_0(HR)\simeq R$. Consequently, we obtain $\mathrm{Pic}(\cD_c(R))\simeq \mathrm{Pic}(R) \times \bbZ$; consult Remark \ref{rcoeff}(i) for the case where $R$ is decomposable. This computation was originally established in \cite{Fausk}. Although Fausk did not use weight structures, one observes that by applying our arguments (see \S\ref{sec:proof2}) to the triangulated category $D_c(R)$, equipped with the weight structure whose heart consists of the finitely generated projective $R$-modules, one obtains a reasoning somewhat similar to his~one.

\end{example}

\section{Weight structures}\label{sec:weight}
In  this section we briefly review the theory of weight structures. This will give us the opportunity to fix some notations that will be used throughout the article.

\begin{definition}{(see \cite[Def.~1.1.1]{Bondarko-Weight})}\label{dws}
A {\em weight structure $w$} on a triangulated category $\cT$, also known as a {\em co-$t$-structure} in the sense of Pauksztello \cite{Pauksztello}, consists of a pair of additive subcategories $(\cT^{w\geq 0}, \cT^{w\leq 0})$ satisfying the following conditions\footnote{Following \cite{Bondarko-Weight}, we will use the so-called {\em cohomological convention} for weight structures. This differs from the homological convention used in \cite{Bondarko-Killing,Bondarko-Gersten,bscwh}.}:
\begin{itemize}
\item[(i)] The categories $\cT^{w\geq 0}$ and $\cT^{w\leq 0}$ are 
Karoubi-closed in $\cT$;
\item[(ii)] We have inclusions of categories $\cT^{w\geq 0}\subset \cT^{w\geq 0}[1]$ and $\cT^{w\leq 0}[1] \subset \cT^{w \leq 0}$;
\item[(iii)] For every $a \in \cT^{w \geq 0}$ and $b \in \cT^{w\leq 0}[1]$, we have $\Hom_\cT(a,b)=0$;
\item[(iv)] For every $a \in \cT$ there exists a distinguished triangle $c[-1]\to a \to b \to c$ in $\cT$ with $b \in \cT^{w\leq 0}$ and $c \in \cT^{w\geq 0}$.
\end{itemize}
\end{definition}
Given an integer $n \in \bbZ$, let $\cT^{w\geq n}:=\cT^{w \geq 0}[-n]$, $\cT^{w\leq n}:=\cT^{w\leq 0}[-n]$, and $\cT^{w =n}:=\cT^{w \geq n} \cap \cT^{w\leq n}$. The objects belonging to $\cup_{n \in \bbZ} \cT^{w=n}$ are called {\em $w$-pure} and the additive subcategory $\cH:=\cT^{w=0}$ is called the {\em heart} of the weight structure. Finally, a weight structure $w$ is called {\em bounded} if $\cT=\cup_{n \in \bbZ}\cT^{w \geq n}= \cup_{n \in \bbZ} \cT^{w\leq n}$.
\medskip
\newline
{\bf Assumption:} Let $(\cT,\otimes, {\bf 1})$ be a symmetric monoidal triangulated category equipped with a weight structure $w$. Throughout the article, we will always assume that the symmetric monoidal structure is {\em $w$-pure} in the sense that the tensor product $-\otimes -$ (as well as the $\otimes$-unit ${\bf 1}$) restricts to the heart $\cH$.

\begin{remark}[Self-duality]\label{rwd} 
The notion of weight structure is (categorically) self-dual. Given a triangulated category $\cT$ equipped with a weight structure $w$, the opposite triangulated category $\cT^\op$ inherits the opposite weight structure $w^\op$ with $(\cT^\op)^{w^\op \leq 0}:= \cT^{w \geq 0}$ and $(\cT^\op)^{w^\op \geq 0}:= \cT^{w \leq 0}$.
%
\end{remark}

\begin{definition}\label{def:weight-exact}
An exact functor $F\colon \cT\to \cT'$ between triangulated categories equipped with weight structures $w$ and $w'$, respectively, is called  {\em weight-exact}
if $F(\cT^{w\le 0})\subseteq  \cT'{}^{w'\le 0}$ and $F(\cT^{w\ge 0})\subseteq  \cT'{}^{w'\ge 0}$. 
\end{definition}
\begin{remark}\label{rk:weight-exact}
Whenever the weight structure $w$ is bounded, the functor $F\colon \cT \to \cT'$ is weight-exact if and only if $F(\cT^{w=0})\subseteq \cT^{w'=0}$; see \cite[Prop.~1.2.3(5)]{bscwh}.
\end{remark}

\subsection{Weight complexes}
Let $\cT$ be a triangulated category equipped with a weight structure $w$. Following \cite[Def.~2.2.1]{Bondarko-Weight}, we can assign to every object $a\in \cT$ a certain (cochain) weight $\cH$-complex $t(a): \cdots \to a^{m-1} \to a^m \to a^{m+1} \to \cdots$. For example, if $a \in \cT^{w=m}$, then we can take for $t(a)$ the complex $\cdots \to 0 \to a \to 0 \to \cdots$ supported in degree $m$. As explained in {\em loc. cit.}, the assignment $a \mapsto t(a)$ is well-defined only up to homotopy equivalence. Nevertheless, we have the result:
\begin{proposition}{(see \cite[Cor. 2.3.4]{Bondarko-Gersten})}\label{prop:homological}
Given an additive functor $G\colon \cH \to \mathrm{A}$, with values in an abelian category, the assignment $a \mapsto H^0(G(t(a)))$ yields a well-defined (\ie independent of the choice of $t(a)$) homological functor $\mathrm{H}_0\colon \cT \to \mathrm{A}$. Moreover, the assignment $G\mapsto \mathrm{H}_0$ is natural in the functor $G$.
\end{proposition}
We denote by $\mathrm{H}_n$ the precomposition of $\mathrm{H}_0$ with the $n^{\mathrm{th}}$ suspension functor of~$\cT$.
\begin{remark}\label{rk:one-way}
Note that if $a \in \cT^{w=m}$, then $\mathrm{H}_n(a)=0$ for every $n \neq m$.
\end{remark}
Let $F\colon \cT \to \cT'$ be a weight-exact functor as in Definition \ref{def:weight-exact}. If $t(a)$ is a weight $\cH$-complex for $a$, then $F(t(a))$ is a weight $\cH'$-complex for $F(a)$.
\subsection{Karoubization}\label{sub:Kar}
Given a category $\cC$, let us write $\mathrm{Kar}(\cC)$ for its Karoubization. Recall that the objects of $\mathrm{Kar}(\cC)$ are the pairs $(a,e)$, with $a \in \cC$ and $e$ an idempotent of the ring of endomorphisms $\Hom_\cC(a,a)$. The morphisms are given by $\Hom_{\mathrm{Kar}(\cC)}((a,e),(b,e')):=e \circ \Hom_\cC(a,b) \circ e'$. By construction, $\mathrm{Kar}(\cC)$ comes equipped with the canonical functor $\cC \to \mathrm{Kar}(\cC), a \mapsto (a,\id)$. Whenever $\cC$ is symmetric monoidal, resp. triangulated, the category $\mathrm{Kar}(\cC)$ is also symmetric monoidal, resp. triangulated (see \cite[Thm.~1.5]{bashli}). Moreover, the canonical functor $\cC \to \mathrm{Kar}(\cC)$ becomes symmetric monoidal, resp. exact. 

Let $\cT$ be a triangulated category equipped with a bounded weight structure $w$. The Karoubization of $\cT^{w \geq 0}$ and $\cT^{w \leq 0}$ inside $\mathrm{Kar}(\cT)$ equip the latter category with a bounded weight structure, making the canonical functor $\cT \to \mathrm{Kar}(\cT)$ weight-exact; consult \cite[Prop. 3.2.1 and Rk. 3.2.2(2)]{Bondarko-Killing} for details.
\section{Proof of Theorem \ref{thm:main1}}
We start with the following auxiliary result:
\begin{proposition}\label{prop:auxiliar}
A symmetric monoidal triangulated category $(\cT, \otimes, {\bf 1})$, equipped with a weight structure $w$, has the $w$-Picard property (see \S\ref{sec:intro}) if and only if all its $\otimes$-invertible objects are $w$-pure.
\end{proposition}
\begin{proof}
Let $(a,n), (b,m) \in \mathrm{Pic}(\cH) \times \bbZ$. On the one hand, when $n=m$, we have $a[n]\simeq b[m]$ in $\cT$ if and only if $a \simeq b$ in $\cH$. This follows from the fact that the suspension functor is an equivalence of $\cT$. On the other hand, when $n\neq m$, we have $a[n]\not\simeq  b[m]$ in $\cT$. This follows from the fact that $\Hom_\cT(a[n],b[m])$, resp. $\Hom_\cT(b[m],a[n])$, is zero whenever $m <n$, resp. $n <m$; see Definition \ref{dws}(iii). This implies that the canonical group homomorphism
\begin{eqnarray}\label{eq:canonical}
\mathrm{Pic}(\cH) \times \bbZ \too \mathrm{Pic}(\cT) && (a,n) \mapsto a[n]
\end{eqnarray} 
is injective. Consequently, we conclude that the category $\cT$ has the $w$-Picard property if and only if \eqref{eq:canonical} is surjective. In other words, $\cT$ has the $w$-Picard property if and only if all its $\otimes$-invertible objects are $w$-pure.
\end{proof}
\begin{remark}
Let $(\cT, \otimes, {\bf 1})$ be a symmetric monoidal triangulated category equipped with a weight structure $w$. The arguments used in the proof of  Proposition \ref{prop:auxiliar} allow us to conclude that if by hypothesis $a[n]\otimes b[m]\simeq {\bf 1}$ for certain objects $a, b \in \cH$ and integers $n, m \in \bbZ$, then $n=-m$ and $a$ is the $\otimes$-inverse of $b$.
\end{remark}
Let us now prove Theorem \ref{thm:main1}. Following \S\ref{sub:Kar}, we can assume without loss of generality that the categories $\cT$ and $\cH$ are Karoubian. Let $b \in \cT$ be a (fixed) $\otimes$-invertible object. Thanks to Proposition \ref{prop:auxiliar}, it suffices to prove that $b$ is $w$-pure. 

As proved in \cite[A.2.10]{AK}, every Karoubian semi-simple category is abelian. Therefore, by applying Proposition \ref{prop:homological} to the identity functor $G=\Id\colon \cH \to \cH$, we obtain well-defined homological functors $\mathrm{H}_n\colon \cT \to \cH, n \in \bbZ$.

Consider the homological functor $\cT \to \cH, a \mapsto  \mathrm{H}_0(a\otimes b)$. Since by assumption the weight structure $w$ is bounded, \cite[Thm.~2.3.2]{Bondarko-Weight} applied to the preceding homological functor yields a convergent K\"unneth spectral sequence
\begin{equation}\label{eq:spectral}
E_1^{pq} = \mathrm{H}_q(a^p \otimes b) \Rightarrow \mathrm{H}_{p+q}(a \otimes b)\,.
\end{equation}
Using the fact that $a^p \in \cH$ and that $a^p \otimes t(b)$ is a weight $\cH$-complex for $a^p \otimes b$, we observe that $\mathrm{H}_q(a^p\otimes b)\simeq a^p \otimes \mathrm{H}_q(b)$. Consequently, the semi-simplicity of the heart $\cH$ implies that\footnote{Whenever the spectral sequence \eqref{eq:spectral} degenerates at the second page, we obtain an induced K\"unneth formula $ \mathrm{H}_n(a\otimes b) \simeq \oplus_{p+q =n} \mathrm{H}_p(a) \otimes \mathrm{H}_q(b)$.} $E_2^{pq}\simeq\mathrm{H}_p(a) \otimes \mathrm{H}_q(b)$. Let us denote by $m_a$, resp. $m'_a$, the smallest, resp. largest, integer such that $\mathrm{H}_n(a)=0$ for every $n < m_a$, resp. $n> m'_a$; the existence of such integers follows from the fact that the weight structure $w$ is bounded. Similarly, let $m_b$, resp. $m'_b$, be the smallest, resp. largest, integer such that $\mathrm{H}_n(b)=0$ for every $n<m_b$, resp. $n > m'_b$. Since by assumption the category $\cH$ is local, we have $\mathrm{H}_{m_a}(a) \otimes \mathrm{H}_{m_b}(b)\neq0$ and $\mathrm{H}_{m'_a}(a) \otimes \mathrm{H}_{m'_b}(b)\neq 0$. Using the second page of the spectral sequence \eqref{eq:spectral}, we conclude that 
\begin{eqnarray}\label{eq:non-zero}
\mathrm{H}_{m_a + m_b}(a\otimes b)\neq 0 & \mathrm{and} & \mathrm{H}_{m'_a + m'_b}(a\otimes b)\neq 0\,.
\end{eqnarray}
Now, assume that $b$ is $\otimes$-invertible. By definition, $a\otimes b \simeq {\bf 1}$ for some ($\otimes$-invertible) object $a \in \cT$. Since $\mathrm{H}_n(a\otimes b)\simeq \mathrm{H}_n({\bf 1})=0$ for every $n \neq 0$, we conclude from \eqref{eq:non-zero} that $m_b=m'_b$, $m_a =m'_a$, and $m_a=-m_b$. Thanks to Proposition \ref{prop:conservativity-1}, this implies that $b \in \cT^{w=m_b}$. In particular, the object $b$ is $w$-pure, and so the proof is finished.
\begin{proposition}{(Conservativity I)}\label{prop:conservativity-1}
Let $\cT$ be a triangulated category equipped with a bounded weight structure $w$ whose heart $\cH$ is abelian semi-simple. Consider the associated homological functors $\mathrm{H}_n\colon \cT \to \cH, n \in \bbZ$. Under these assumptions, an object $b\in \cT$ belongs to $\cT^{w=m}$ if and only if $\mathrm{H}_n(b)=0$ for every $n\neq m$.
\end{proposition}
\begin{proof}
Consult \cite[Rk.~3.3.6]{Bondarko-Killing}.
\end{proof}

\begin{remark}[K\"unneth spectral sequence]
Let $(\cT,\otimes, {\bf 1})$ be a symmetric monoidal triangulated equipped with a bounded weight structure $w$, and $G\colon \cH \to \mathrm{A}$ a symmetric monoidal additive functor. Consider the associated homological functors $\mathrm{H}_n\colon \cT \to \mathrm{A}, n \in \bbZ$. The arguments used in the proof of Theorem \ref{thm:main1} allow us to conclude that there exists a convergent K\"unneth spectral sequence 
$$E_1^{pq}=\mathrm{H}_q(a^p\otimes b) \Rightarrow \mathrm{H}_{p+q}(a\otimes b)\,.$$
Assume that the (abelian) category $\mathrm{A}$ is moreover semi-simple and local. Then, given any $\otimes$-invertible object $b \in \cT$, there exists an integer $m_b$ such that $\mathrm{H}_n(b)=0$ for every $n \neq m_b$ and $\mathrm{H}_{m_b}(b)\in \mathrm{A}$ is $\otimes$-invertible.
\end{remark}
\section{Proof of Theorem \ref{thm:main2}}\label{sec:proof2}



Let $b \in \cT$ be a $\otimes$-invertible object. Thanks to Proposition \ref{prop:auxiliar}, it suffices to prove that $b$ is $w$-pure. Since the functors $\iota_{\kappa(\mathfrak{p})}\colon \cT \to \cT_{\kappa(\mathfrak{p})}$ are symmetric monoidal, and by assumption the categories $\cT_{\kappa(\mathfrak{p})}$ have the $w_{\kappa(\mathfrak{p})}$-Picard property, the objects $\iota_{\kappa(\mathfrak{p})}(b)$ are $w_{\kappa(\mathfrak{p})}$-pure. Concretely, $\iota_{\kappa(\mathfrak{p})}(b)$ belongs to $\cT_{\kappa(\mathfrak{p})}^{w=m_{\kappa(\mathfrak{p})}}$ for some integer $m_{\kappa(\mathfrak{p})}\in \bbZ$. Our goal is to prove that all the integers $m_{\kappa(\mathfrak{p})}$, with $\mathfrak{p} \in \mathrm{Spec}(R)$, are equal and that the object $b$ belongs to $ \cT^{w=m_{k(\mathfrak{p})}}$.


We start by addressing the first goal. Since by assumption the commutative ring $R$ is indecomposable, its spectrum $\mathrm{Spec}(R)$ is connected. Hence, it suffices to verify that $m_{\kappa(\mathfrak{p})}=m_{\kappa(\mathfrak{P})}$ for every $\mathfrak{p} \in \mathrm{Spec}(R)$ belonging to the closure of a prime ideal $\mathfrak{P} \in \mathrm{Spec}(R)$; in the particular case where $R$ is moreover an integral domain we can simply take $\mathfrak{P}=\{0\}$. Note that the assumptions of Theorem \ref{thm:main2}, as well as the definition of the integers $m_{\kappa(\mathfrak{p})}$ and $m_{\kappa(\mathfrak{P})}$,  
are (categorically) self-dual; see Remark \ref{rwd}. Therefore, it is enough to verify the inequalities $m_{\kappa(\mathfrak{p})}\ge m_{\kappa(\mathfrak{P})}$.

 
Given an $R$-algebra $S$, consider the abelian category $\mathrm{PShv}^S(\cH)$ of $R$-linear functors from $\cH^\op$ to the category of $S$-modules. Note that the Yoneda functor
\begin{eqnarray}\label{eq:Yoneda}
\cH \too \mathrm{PShv}^S(\cH) && a \mapsto \left(c \mapsto \Hom_\cH(c,a) \otimes_R S\right)
\end{eqnarray} 
induces a fully-faithful embedding of $\cH\otimes_R S$ into the full subcategory of $\mathrm{PShv}^S(\cH)$ consisting of projective objects; see \cite[Lem.~8.1]{MVW}. Note also that every $R$-algebra homomorphism $S \to S'$ gives rise to a functor $-\otimes_S S'\colon \mathrm{PShv}^S(\cH) \to \mathrm{PShv}^{S'}(\cH)$. Since $\mathrm{PShv}^S(\cH)$ is abelian, Proposition \ref{prop:homological} yields an homological~functor
\begin{eqnarray*}
\mathrm{H}_0^S\colon \cT \too \mathrm{PShv}^S(\cH) && a \mapsto \left(c \mapsto H^0(\Hom_\cH(c, t(a))\otimes_R S)\right)\,.
\end{eqnarray*}
Recall from assumption (A2) that the functor $\iota_{\kappa(\mathfrak{p})}$ induces a $\otimes$-equivalence of categories $\mathrm{Kar}(\cH \otimes_R \kappa(\mathfrak{p}))\simeq \cH_{\kappa(\mathfrak{p})}$. This implies that $\mathrm{H}_0^{\kappa(\mathfrak{p})}$ factors through $\iota_{\kappa(\mathfrak{p})}$. Consequently, thanks to Remark \ref{rk:one-way}, we have $\mathrm{H}_n^{\kappa(\mathfrak{p})}(b)=0$ for every $n \neq m_{\kappa(\mathfrak{p})}$. 

Let us denote by $Q$ the localization of $R/\mathfrak{P}$ at the prime ideal $\mathfrak{p}$. Note that $Q$ is a local  Noetherian integral domain with fraction field $\kappa(\mathfrak{P})$.
Recall from assumption (A1) that the commutative ring $R$ is Noetherian and that the $R$-modules of morphisms of the heart $\cH$ are finitely generated and flat. 
Thanks to the universal coefficients theorem, this implies that 
$\mathrm{H}_l^{Q}(b)\otimes_{Q}\kappa(\mathfrak{p}) 
=\mathrm{H}_l^{\kappa(\mathfrak{p})}(b)$, with $l$ being the largest integer such that $\mathrm{H}_l^{Q}(b)\neq 0$. 
Consequently, by applying the Nakayama lemma to the local ring $Q$ and to the (objectwise) finitely generated $Q$-module $\mathrm{H}^{Q}_l(b)$, we conclude that $\mathrm{H}^{\kappa(\mathfrak{p})}_l(b)\neq 0$. Hence, the equality $m_{\kappa(\mathfrak{p})}=l$ holds. Now, since $\kappa(\mathfrak{P})$ is a flat $Q$-module, the universal coefficients theorem yields that $\mathrm{H}_n^{\kappa(\mathfrak{P})}(b)=0$ for every $n>l$. This allows us to conclude that $l=m_{k(\mathfrak{p})}\geq m_{k(\mathfrak{P})}$.
 
Let us now address the second goal, \ie prove that $b \in \cT^{w=m}$ with $m:=m_{k(\mathfrak{p})}$. Making use of Remark \ref{rwd} once again, we observe that it suffices to prove that $b \in \cT^{w\leq m}$. Thanks to Proposition \ref{prop:conservativity-2}, it is enough to verify that $\mathrm{H}_n^R(b)=0$ for every $n>m$. Let us denote by $l$ the largest integer such that $\mathrm{H}_l^R(b)\neq 0$. An argument similar to the one used in the preceding paragraph, implies that $\mathrm{H}_l^R(b)\otimes_R \kappa(\mathfrak{p})= \mathrm{H}_l^{\kappa(\mathfrak{p})}(b)$ for every $\mathfrak{p}\in  \mathrm{Spec}(R)$. Since $\mathrm{H}_n^{\kappa(\mathfrak{p})}(b)=0$ for all $n>m$ and $\mathfrak{p} \in \mathrm{Spec}(R)$, we then conclude that $\mathrm{H}_n^R(b)=0$ for every $n>m$. This finishes the proof.


\begin{proposition}[Conservativity II]\label{prop:conservativity-2}
Let $\cT$ be a triangulated category~equipped~with a bounded weight structure $w$ whose heart $\cH$ is $R$-linear and small. Consider the associated homological functors $\mathrm{H}^R_n\colon \cT \to \mathrm{PShv}^R(\cH), n \in \bbZ$. Under these assumptions, an object $b \in \cT$ belongs to $\cT^{w\le m}$ if and only if $\mathrm{H}^R_n(b)=0$ for every $n > m$.
\end{proposition}
\begin{proof}
Combine \cite[Prop.~3.3.3]{Bondarko-Killing} with \cite[Rk.~3.3.4(2)]{Bondarko-Killing}.
\end{proof}

\begin{remark}\label{rcoeff} 
\begin{itemize}
\item[(i)] Suppose that in Theorem \ref{thm:main2} the commutative ring $R$ is of the form $\Pi_{j=1}^r R_j$, with $R_j$ an indecomposable Noetherian ring. In this case, the corresponding idempotents $e_j \in R$ give naturally rise to categorical decompositions $\cT\simeq \Pi^r_{j=1} \cT_j$ and $\cH \simeq \Pi_{j=1}^r \cH_j$. By applying Theorem \ref{thm:main2} to each one of the categories $\cT_j$, we conclude that
$$ \mathrm{Pic}(\cT) \simeq \Pi^r_{j=1} \mathrm{Pic}(\cT_j) \simeq \Pi^r_{j=1}(\mathrm{Pic}(\cH_j)\times \bbZ) \simeq \mathrm{Pic}(\cH) \times \bbZ^r$$
whenever all the triangulated categories $\cT_j$ are non-zero;


\item[(ii)] At assumption (A2) of Theorem \ref{thm:main2}, instead of working with all prime ideals $\mathfrak{p} \in \mathrm{Spec}(R)$, note that it suffices to consider any connected subset of $\mathrm{Spec}(R)$ that contains all maximal ideals of $R$. For example, in the particular case where $R$ is local, it suffices to consider the (unique) closed point $\mathfrak{p_0}$ of $\mathrm{Spec}(R)$.

\end{itemize}
\end{remark}

\section{Proof of Theorem \ref{thm:computation-2}}\label{sub:proof22}
Recall from \cite[Part 4 and Lecture~20]{MVW}\cite{Voevodsky} the construction of the symmetric monoidal triangulated category $\mathrm{DM}_{\mathrm{gm}}(k;R)$. Given any two zero-dimensional smooth $k$-schemes $X$ and $Y$, we have trivial positive Ext-groups:
\begin{eqnarray*}
\Hom_{\mathrm{DMA}(k;R)}(M(X)_R, M(Y)_R[n])=0 && n >0\,.
\end{eqnarray*}
This implies that the subcategory $\mathrm{AM}(k;R)\subset \mathrm{DMA}(k;R)$ is {\em negative} in the sense of \cite[Def.~4.3.1(1)]{Bondarko-Weight}. Making use of \cite[Thm.~4.3.2 II and Prop.~5.2.2]{Bondarko-Weight}, we conclude that the category $\mathrm{DMA}(k;R)$ carries a bounded weight structure with heart $\mathrm{AM}(k;R)$. By restriction, $\mathrm{D}\cA(k;R)$ inherits a bounded weight structure $w_R$~with~heart~$\cA(k;R)$.

Let us now show that the category $\mathrm{D}\cA(k;R)$ has the $w_R$-Picard property; note that this automatically concludes the proof. By construction, $\cA(k;R)$ is essentially small. Moreover, we have natural isomorphisms
$$\Hom_{\mathrm{D}\cA(k;R)}(M(X)_R,M(Y)_R)\simeq CH^0(X\times Y)_R\,.$$ 
Since the $R$-modules $CH^0(X\times Y)_R$ are free, assumptions (A1) of Theorem \ref{thm:main2} are verified. In what concerns assumptions (A2), take for $\cT_{\kappa(\mathfrak{p})}$ the category $\mathrm{D}\cA(k;\kappa(\mathfrak{p}))$ and for $\iota_{\kappa(\mathfrak{p})}$ the functor $-\otimes_R \kappa(\mathfrak{p})\colon \mathrm{D}\cA(k;R) \to \mathrm{D}\cA(k;\kappa(\mathfrak{p}))$. By construction, the latter functor is weight-exact (see Remark \ref{rk:weight-exact}), symmetric monoidal, and induces a $\otimes$-equivalence of categories between $\mathrm{Kar}(\cA(k;R) \otimes_R \kappa(\mathfrak{p}))$ and $\cA(k;\kappa(\mathfrak{p}))$. This shows that assumptions (A2) are also verified. 

Let us now prove that the categories $\mathrm{D}\cA(k;\kappa(\mathfrak{p}))$ have the $w_{\kappa(\mathfrak{p})}$-Picard property; thanks to Theorem \ref{thm:main2} this implies that $\mathrm{D}\cA(k;R)$ has the $w_R$-Picard property. In order to do so, we will make use of Theorem \ref{thm:main1}. Concretely, we need to prove that the categories $\cA(k;\kappa(\mathfrak{p}))$ are semi-simple and local. Let us write $L$ for the composite of the finite separable field extensions $l_i/k$ inside $\overline{k}$, $\mathrm{G}$ for the profinite Galois group $\mathrm{Gal}(L/k)$, and $\mathrm{G}_i$ for the finite Galois group $\mathrm{Gal}(l_i/k)$. Thanks to assumption (B1), there is a $\otimes$-equivalence between $\cA(k;\kappa(\mathfrak{p}))$ and the category of finite dimensional $\kappa(\mathfrak{p})$-linear continuous $\mathrm{G}$-representations $\mathrm{Rep}_{\kappa(\mathfrak{p})}(\mathrm{G})$. Consequently, since $\mathrm{G}\simeq \mathrm{lim}_i \mathrm{G}_i$, we conclude that $\cA(k;\kappa(\mathfrak{p}))\simeq\mathrm{colim}_i \mathrm{Rep}_{\kappa(\mathfrak{p})}(\mathrm{G}_i)$. Now, since the group $\mathrm{G}_i$ is finite, the category $\mathrm{Rep}_{\kappa(\mathfrak{p})}(\mathrm{G}_i)$ may be identified with the category of finitely generated (right) $\kappa(\mathfrak{p})[\mathrm{G}_i]$-modules. Thanks to assumption (B2), the degree of the field extension $l_i/k$ is invertible in $R$ and hence in $\kappa(\mathfrak{p})$. The (classical) Maschke theorem then implies that the category $\mathrm{Rep}_{\kappa(\mathfrak{p})}(\mathrm{G}_i)$ is semi-simple. Note that this category is moreover local since the tensor product is defined on the underlying $\kappa(\mathfrak{p})$-vector spaces. The proof follows now automatically from the above description of the categories $\cA(k;\kappa(\mathfrak{p}))$.
\section{Proof of Theorem \ref{thm:computation-4}}
Let us denote by $\cA\mathrm{T}(k;R)$ the smallest additive, Karoubian, full subcategory of $\mathrm{D}\cA\mathrm{T}(k;R)$ containing the objects $M(X)_R(m)[2m]$, with $M(X)_R \in \cA$ and $m \in \bbZ$. Under these notations, we have trivial positive Ext-groups:
\begin{eqnarray*}
\Hom_{\mathrm{D}\cA\mathrm{T}(k;R)}(M(X)_R(m)[2m], M(Y)_R(m')[2m'][n])=0 && n >0\,.
\end{eqnarray*}
This implies that the subcategory $\cA\mathrm{T}(k;R) \subset \mathrm{D}\cA\mathrm{T}(k;R)$ is negative. Making use of \cite[Thm.~4.3.1 II and Prop.~5.2.2]{Bondarko-Weight}, we conclude that $\mathrm{D}\cA\mathrm{T}(k;R)$ carries a bounded weight structure $w_R$ with heart $\cA\mathrm{T}(k;R)$. Thanks to the equivalence of categories
\begin{eqnarray*}
\mathrm{Gr}_\bbZ\cA(k;R) \stackrel{\simeq}{\too} \cA\mathrm{T}(k;R) && \{M(X_m)\}_{m \in \bbZ} \mapsto \bigoplus_{m \in \bbZ} M(X_m)(m)[2m]\,,
\end{eqnarray*}
an argument similar to the one of the proof of Theorem \ref{thm:computation-2} implies that the category $\mathrm{D}\cA\mathrm{T}(k;R)$ has the $w_R$-Picard property. Consequently, we have $\mathrm{Pic}(\mathrm{D}\cA\mathrm{T}(k;R))\simeq \mathrm{Pic}(\cA\mathrm{T}(k;R))\times \bbZ$. The proof follows now from the natural isomorphisms 
$$\mathrm{Pic}(\cA\mathrm{T}(k;R))\simeq \mathrm{Pic}(\mathrm{Gr}_\bbZ\cA(k;R))\simeq \mathrm{Pic}(\cA(k;R))\times \bbZ\,.$$
\section{Proof of Theorem \ref{thm:computation-last}}
Recall from Ayoub \cite[\S4]{Ayoub}\cite[\S2.1.1]{Ayoub2} the construction of the symmetric monoidal triangulated category ${\bf DA}(k;\bbZ)$; we write $\mathrm{Boot}(k;\bbZ)$ for the thick triangulated subcategory generated by the $\otimes$-unit $\Sigma^\infty(\mathrm{Spec}(k)_+)_\bbZ$. By construction, we have exact symmetric monoidal functors $(-)_\bbZ\colon \mathrm{SH}(k) \to {\bf DA}(k;\bbZ)$ which restrict to the bootstrap categories. Let $\mathrm{P}(k)$, resp. $\mathrm{P}(k;\bbZ)$, be the smallest additive, Karoubian, full subcategory of $\mathrm{Boot}(k)$, resp. $\mathrm{Boot}(k;\bbZ)$, containing the $\otimes$-unit $\Sigma^\infty(\mathrm{Spec}(k)_+)$, resp. $\Sigma^\infty(\mathrm{Spec}(k)_+)_\bbZ$. We have trivial positive Ext-groups (see  \cite[Thm.~4.14]{ICM}):
$$ \Hom_{\mathrm{Boot}(k)}(\Sigma^\infty(\mathrm{Spec}(k)_+), \Sigma^\infty(\mathrm{Spec}(k)_+)[n])=0 \quad n>0\,;$$
similarly for $\mathrm{Boot}(k;\bbZ)$. This implies that the subcategory $\mathrm{P}(k)\subset \mathrm{Boot}(k)$, resp. $\mathrm{P}(k;\bbZ) \subset \mathrm{Boot}(k;\bbZ)$, is negative. Making use of \cite[Thm.~4.3.2 II and Prop.~5.2.2]{Bondarko-Weight}, we conclude that the category $\mathrm{Boot}(k)$, resp. $\mathrm{Boot}(k;\bbZ)$, carries a bounded weight structure $w$, resp. $w_\bbZ$, with heart $\mathrm{P}(k)$, resp. $\mathrm{P}(k;\bbZ)$.

Let us now show that the category $\mathrm{Boot}(k)$ has the $w$-Picard property. Thanks to the ring isomorphism \eqref{eq:ring-iso}, $\mathrm{P}(k)$ identifies with the category $\mathrm{Proj}(GW(k))$ of finitely generated projective $GW(k)$-modules. Consequently, since the Grothendieck-Witt ring $GW(k)$ is indecomposable (see \cite[Prop.~2.2]{Witt}), all the assumptions (A1) of Theorem \ref{thm:main2} (with $R=GW(k)$) are verified. In what concerns assumptions (A2), take for $\cT_{\kappa(\mathfrak{p})}$ the bounded derived category $\cD^b(\kappa(\mathfrak{p}))$ of finite dimensional $\kappa(\mathfrak{p})$-vector spaces $\mathrm{Vect}(\kappa(\mathfrak{p}))$ and for $\iota_{\kappa(\mathfrak{p})}$ the composed functor:
\begin{equation}\label{eq:composed}
\mathrm{Boot}(k) \stackrel{(-)_\bbZ}{\too} \mathrm{Boot}(k;\bbZ) \stackrel{t(-)}{\too} K^b(\mathrm{Proj}(GW(k))) \stackrel{-\otimes_{GW(k)}\kappa(\mathfrak{p})}{\too} \cD^b(\kappa(\mathfrak{p}))\,.
\end{equation}
Some explanations are in order: since the category ${\bf DA}(k;\bbZ)$ is defined as the localization of a certain category of complexes, it admits a tensor differential graded (=dg) enhancement. Making use of \cite[Lem.~18]{bach}, we then conclude that the weight complex construction gives rise to an exact symmetric monoidal functor $t(-)$ with values in the bounded homotopy category of $\mathrm{Proj}(GW(k))$. By construction, the composed functor \eqref{eq:composed} is weight-exact, symmetric monoidal, and induces a $\otimes$-equivalence of categories between $\mathrm{Kar}(\mathrm{P}(k) \otimes_{GW(k)} \kappa(\mathfrak{p}))$ and $\mathrm{Vect}(\kappa(\mathfrak{p}))$. This shows that the assumptions (A2) are also verified. Finally, since the categories $\cD^b(\kappa(\mathfrak{p}))$ clearly have the $w_{\kappa(\mathfrak{p})}$-Picard property, we conclude from Theorem \ref{thm:main2} that $\mathrm{Boot}(k)$ has the $w$-Picard property. This finishes the proof.

\begin{remark}[Coefficients]
Let $R$ be a commutative ring. Instead of ${\bf DA}(k;\bbZ)$ and $\mathrm{Boot}(k;\bbZ)$, we can consider more generally the symmetric monoidal triangulated categories ${\bf DA}(k;R)$ and $\mathrm{Boot}(k;R)$, respectively. Under the corresponding conditions, a proof similar to the one of Theorem \ref{thm:computation-last}, with $\mathrm{Boot}(k)$ and $GW(k)$ replaced by $\mathrm{Boot}(k;R)$ and $GW(k)_R$, shows that $\mathrm{Pic}(\mathrm{Boot}(k;R)) \simeq \mathrm{Pic}(GW(k)_R) \times \bbZ$. This implies, in particular, that the categories $\mathrm{Boot}(k)$ and $\mathrm{Boot}(k;\bbZ)$, although not equivalent, have nevertheless the same Picard group!
\end{remark}

\section{Proof of Theorem \ref{thm:computation-6}}
Recall from \cite[\S9]{book}\cite[\S4]{Hopf} the construction of the symmetric monoidal triangulated category $\mathrm{KMM}(k;R)$. Given any two finite separable field extensions $l_1/k$ and $l_2/k$, we have trivial positive Ext-groups (see \cite[Prop.~4.4]{Hopf}):
\begin{eqnarray*}
\Hom_{\mathrm{NMAM}(k;R)}(U(l_1)_R, U(l_2)_R[n])\simeq \pi_{-n}(K(l_1\otimes_k l_2)\wedge HR) =0 && n >0\,.
\end{eqnarray*}
This implies that the subcategory $\mathrm{AM}(k;R) \subset \mathrm{NMAM}(k;R)$ is negative. Making use of \cite[Thm.~4.3.1 II and Prop.~5.2.2]{Bondarko-Weight}, we conclude that the category $\mathrm{NMAM}(k;R)$ carries a bounded weight structure\footnote{A bounded weight structure on the category of noncommutative mixed motives was originally constructed in \cite[Thm.~1.1]{Weight}.} with heart $\mathrm{AM}(k;R)$. By restriction, $\mathrm{NM}\cA(k;R)$ inherits a bounded weight structure $w_R$ with heart $\cA(k;R)$.

Now, a proof similar to the one of Theorem \ref{thm:computation-2}, with $\mathrm{D}\cA(k;R)$ and $\mathrm{D}\cA(k;\kappa(\mathfrak{p}))$ replaced by $\mathrm{NM}\cA(k;R)$ and $\mathrm{NM}\cA(k;\kappa(\mathfrak{p}))$, respectively, allows us to conclude that the category $\mathrm{NM}\cA(k;R)$ has the $w_R$-Picard property. Consequently, we have $\mathrm{Pic}(\mathrm{NM}\cA(k;R))\simeq \mathrm{Pic}(\cA(k;R))\times \bbZ$.  

\section{Proof of Theorem \ref{thm:computation-7}}\label{sec:prooflast}
\subsection*{Item (i)} Similarly to the proof of Theorem \ref{thm:computation-6}, given any two central simple $k$-algebras $A$ and $B$, we have trivial positive Ext-groups (see \cite[Prop.~4.4]{Hopf}):
\begin{eqnarray*}
\Hom_{\mathrm{NMCSA}(k;R)}(U(A)_R, U(B)_R[n])\simeq \pi_{-n} (K(A^\op \otimes_k B) \wedge HR)=0 && n >0\,.
\end{eqnarray*}
This implies that the subcategory $\mathrm{CSA}(k;R) \subset \mathrm{NMCSA}(k;R)$ is negative. Making use of \cite[Thm.~4.3.2 II and Prop.~5.2.2]{Bondarko-Weight}, we conclude that the category $\mathrm{NMCSA}(k;R)$ carries a bounded weight structure $w_R$ with heart $\mathrm{CSA}(k;R)$. 

Let us now show that the category $\mathrm{NMCSA}(k;R)$ has the $w_R$-Picard property. By construction, the category $\mathrm{CSA}(k;R)$ is essentially small. Furthermore, as explained in the proof of \cite[Prop.~2.25]{separable}, we have natural identifications ($\mathrm{ind}$=index)
\begin{equation}\label{eq:identifications}
\Hom_{\mathrm{CSA}(k;R)}(U(A)_R, U(B)_R)= \mathrm{ind}(A^\op \otimes_k B) \cdot R\,,
\end{equation}
under which the composition law of $\mathrm{CSA}(k;R)$ corresponds to multiplication. This implies, in particular, that the assumptions (A1) of Theorem \ref{thm:main2} are verified. In what concerns assumptions (A2), take for $\cT_{\kappa(\mathfrak{p})}$ the category $\mathrm{NMCSA}(k;\kappa(\mathfrak{p}))$ and for $\iota_{\kappa(\mathfrak{p})}$ the functor $-\otimes_R \kappa(\mathfrak{p})\colon \mathrm{NMCSA}(k;R) \to \mathrm{NMCSA}(k;\kappa(\mathfrak{p}))$. By construction, the latter functor is weight-exact (see Remark \ref{rk:weight-exact}), symmetric monoidal, and induces a $\otimes$-equivalence of categories between $\mathrm{Kar}(\mathrm{CSA}(k;R) \otimes_R \kappa(\mathfrak{p}))$ and $\mathrm{CSA}(k;\kappa(\mathfrak{p}))$. This shows that the assumptions (A2) are also verified.

Let us now prove that the categories $\mathrm{NMCSA}(k;\kappa(\mathfrak{p}))$ have the $w_{\kappa(\mathfrak{p})}$-Picard property; thanks to Theorem \ref{thm:main2} this implies that $\mathrm{NMCSA}(k;R)$ has the $w_R$-Picard property. In order to do so, we will make use of Theorem \ref{thm:main1}. Concretely, we need to prove that $\mathrm{CSA}(k,\kappa(\mathfrak{p}))$ is semi-simple and local. This follows from the next~result:
\begin{proposition}\label{prop:coefficients}
\begin{itemize}
\item[(i)] When $\mathrm{char}(\kappa(\mathfrak{p}))=0$, the category $\mathrm{CSA}(k;\kappa(\mathfrak{p}))$ is $\otimes$-equivalent to the category $\mathrm{Vect}(\kappa(\mathfrak{p}))$ of finite dimensional $\kappa(\mathfrak{p})$-vector spaces;
\item[(ii)] When $\mathrm{char}(\kappa(\mathfrak{p}))=p>0$, the category $\mathrm{CSA}(k;\kappa(\mathfrak{p}))$ is $\otimes$-equivalent to the category $\mathrm{Gr}_{\mathrm{Br}(k)\{p\}}\mathrm{Vect}(\kappa(\mathfrak{p}))$ of $\mathrm{Br}(k)\{p\}$-graded objects in $\mathrm{Vect}(\kappa(\mathfrak{p}))$, where $\mathrm{Br}(k)\{p\}$ stands for the $p$-primary component of $\mathrm{Br}(k)$.
\end{itemize}
\end{proposition}
\begin{proof}
Given any two central simple $k$-algebras $A$ and $B$, we have 
\begin{equation}\label{eq:ident-proof}
\Hom_{\mathrm{CSA}(k;\kappa(\mathfrak{p}))}(U(A)_{\kappa(\mathfrak{p})}, U(B)_{\kappa(\mathfrak{p})})=\mathrm{ind}(A^\op \otimes_k B) \cdot \kappa(\mathfrak{p})\,.
\end{equation}

(i) When $\mathrm{char}(\kappa(\mathfrak{p}))=0$, the right-hand side of \eqref{eq:ident-proof} equals $\kappa(\mathfrak{p})$. This implies that $U(A)_{\kappa(\mathfrak{p})}\simeq U(k)_{\kappa(\mathfrak{p})}$ for every central simple $k$-algebra $A$, and consequently that the category $\mathrm{CSA}(k;\kappa(\mathfrak{p}))$ is $\otimes$-equivalent to $\mathrm{Vect}(\kappa(\mathfrak{p}))$.

(ii) When $\mathrm{char}(\kappa(\mathfrak{p}))=p>0$, the right-hand side of \eqref{eq:ident-proof} equals
\begin{equation}\label{eq:computation-CSA}
\mathrm{ind}(A^\op \otimes_k B) \cdot \kappa(\mathfrak{p})= \begin{cases} \kappa(\mathfrak{p}) & \mathrm{if}\quad p \nmid \mathrm{ind}(A^\op \otimes_k B) \\
0 & \mathrm{if}\quad p \mid\mathrm{ind}(A^\op \otimes_k B)\,.
\end{cases}
\end{equation}
This implies that $U(A)_{\kappa(\mathfrak{p})}\simeq U(k)_{\kappa(\mathfrak{p})}$ for every central simple $k$-algebra $A$ such that $[A] \in \bigoplus_{q \neq p} \mathrm{Br}(k)\{q\}$. Now, let $A$ and $B$ be central simple $k$-algebras such that $[A],[B] \in \mathrm{Br}(k)\{p\}$. Since $\mathrm{ind}(A^\op \otimes_k B)\mid \mathrm{ind}(A^\op) \cdot \mathrm{ind}(B)$, the preceding computation \eqref{eq:computation-CSA} implies also that
$$ U(A)_{\kappa(\mathfrak{p})}\simeq U(B)_{\kappa(\mathfrak{p})}\Leftrightarrow \mathrm{ind}(A^\op \otimes_k B)=1 \Leftrightarrow [A]=[B] \in \mathrm{Br}(k)\{p\}\,.$$
This allows us to conclude that $\mathrm{CSA}(k;\kappa(\mathfrak{p}))\simeq \mathrm{Gr}_{\mathrm{Br}(k)\{p\}}\mathrm{Vect}(\kappa(\mathfrak{p}))$.
\end{proof}
\subsection*{Item (ii)} Thanks to equivalence \eqref{eq:equivalence}, we have an injective group homomorphism
\begin{eqnarray}\label{eq:injective}
\mathrm{Br}(k) \too \mathrm{Pic}(\mathrm{CSA}(k;\bbZ)) && [A] \mapsto U(A)_\bbZ\,.
\end{eqnarray}
Using \cite[Thm.~2.19(iv)]{separable}, we observe that the objects $U(A_1)_\bbZ \oplus \cdots \oplus U(A_m)_\bbZ$ of $\mathrm{CSA}(k;\bbZ)$ with $m>1$ are {\em not} $\otimes$-invertible. Since the category $\mathrm{CSA}(k;\bbZ)$ is Karoubian (see \cite[Thm.~2.19]{separable}), we then conclude that \eqref{eq:injective} is moreover surjective.
\begin{remark}[Coefficients II]
Recall from \cite[\S9]{book}\cite{Weight} the construction of the symmetric monoidal triangulated category $\mathrm{KMM}(k)$; and hence of the full subcategories $\mathrm{NMCSA}(k)$ and $\mathrm{CSA}(k)$. By construction, we have an exact symmetric monoidal functor $(-)_\bbZ\colon \mathrm{KMM}(k) \to \mathrm{KMM}(k;\bbZ)$ which restricts to a $\otimes$-equivalence $\mathrm{CSA}(k)\simeq \mathrm{CSA}(k;\bbZ)$. A proof similar to the one of Theorem \ref{thm:computation-7}, with $\mathrm{NMCSA}(k;\bbZ)$ replaced by $\mathrm{NMCSA}(k)$, allows us then to conclude that $\mathrm{Pic}(\mathrm{NMCSA}(k))\simeq \mathrm{Br}(k) \times \bbZ$. In conclusion, although the categories $\mathrm{NMCSA}(k)$ and $\mathrm{NMCSA}(k;\bbZ)$ are not equivalent, they have nevertheless the same Picard group!
\end{remark}
\section{Proof of Theorem \ref{thm:computation-9}}\label{sec:prooflast1}
Similarly to the proof of Theorem \ref{thm:computation-7}, with central simple algebras replaced by separable algebras, we observe that the category $\mathrm{NMS}\cA(k;R)$ carries a bounded weight structure $w_R$ with heart $\mathrm{S}\cA(k;R)$. 

Let us now present explicit models of the categories $\mathrm{AM}(k;R)$ and $\mathrm{Sep}(k;R)$, and consequently of the categories $\cA(k;R)$ and $\mathrm{S}\cA(k;R)$. Recall from \S\ref{sec:applications} that $\Gamma:=\mathrm{Gal}(\overline{k}/k)$. The {\em convolution category $\mathrm{Cov}(\Gamma;R)$} is defined as follows: the objects are the finite $\Gamma$-sets $S$; the morphisms $\Hom_{\mathrm{Cov}(\Gamma;R)}(S_1,S_2)$ are the $\Gamma$-invariant functions $\alpha\colon S_1 \times S_2 \to R$; the composition law is the convolution product
$$ \Hom_{\mathrm{Cov}(\Gamma;R)}(S_1,S_2) \times \Hom_{\mathrm{Cov}(\Gamma;R)}(S_2,S_3)\too \Hom_{\mathrm{Cov}(\Gamma;R)}(S_1,S_3)$$
given by 
$$ (\alpha,\beta) \mapsto (\alpha \ast \beta)(s_1, s_3):= \sum_{s_2 \in S_2}\alpha(s_1,s_2) \cdot \beta(s_2, s_3)\,;$$
and the identities are the $\Gamma$-invariant functions $S\times S \to R$ which are equal to $1$ on the diagonal and $0$ otherwise. The disjoint union and the cartesian product of finite $\Gamma$-sets makes $\mathrm{Cov}(\Gamma;R)$ into an additive symmetric monoidal category. As proved in \cite[Prop.~2.3]{separable}, the assignment $S \mapsto U(k_S)_R$, where $k_S$ stands for the commutative separable $k$-algebra $\Hom_\Gamma(S, \overline{k})$, gives rise to a $\otimes$-equivalence of categories between $\mathrm{Kar}(\mathrm{Cov}(\Gamma; R))$ and $\mathrm{AM}(k;R)$.

Given a finite $\Gamma$-set $S$, an element $s \in S$, and an Azumaya $k_S$-algebra $A$, let us write $k_s$ for the finite separable field extension $\Hom_\Gamma(\Gamma s, \overline{k})$, $A_s$ for the central simple $k_s$-algebra $A\otimes_{k_S} k_s$, and $\mathrm{ind}_s(A)$ for the index of $A_s$. The category $\mathrm{Cov}'(\Gamma;R)$ is defined as follows: the objects are the pairs $(S,A)$ where $S$ is a finite $\Gamma$-set and $A$ is an Azumaya $k_S$-algebra; the morphisms $\Hom_{\mathrm{Cov}'(\Gamma;R)}((S_1,A), (S_2, B))$ are the $\Gamma$-invariant functions $\alpha\colon S_1 \times S_2 \to R$ such that $\alpha((s_1,s_2))\in \mathrm{ind}_{(s_1,s_2)}(A^\op \otimes_k B) \cdot R$ for every $(s_1, s_2) \in S_1 \times S_2$; and the composition law and the identities are those of $\mathrm{Cov}(\Gamma;R)$. The direct sum $(S_1,A) \oplus(S_2,B) :=(S_1 \amalg S_2, A\times B)$ and tensor product $(S_1,A) \otimes (S_2,B) := (S_1 \times S_2, A\otimes B)$ make $\mathrm{Cov}'(\Gamma; R)$ into an additive symmetric monoidal category. As proved in \cite[Thm.~2.12]{separable}, the assignment $(S,A)\mapsto U(A)_R$ gives rise to a $\otimes$-equivalence of categories between $\mathrm{Kar}(\mathrm{Cov}'(\Gamma;R))$ and $\mathrm{Sep}(k;R)$.


Let us now show that the category $\mathrm{NMS}\cA(k;R)$ has the $w_R$-Picard property. By construction, $\mathrm{S}\cA(k;R)$ is essentially small. The above model of $\mathrm{Sep}(k;R)$, and of $\mathrm{S}\cA(k;R)$, implies that the remaining assumptions (A1) of Theorem \ref{thm:main2} are also verified. In what concerns assumptions (A2), take for $\cT_{\kappa(\mathfrak{p})}$ the category $\mathrm{NMS}\cA(k;\kappa(\mathfrak{p}))$ and for $\iota_{\kappa(\mathfrak{p})}$ the functor $-\otimes_R \kappa(\mathfrak{p})\colon \mathrm{NMS}\cA(k;R) \to \mathrm{NMS}\cA(k;\kappa(\mathfrak{p}))$. By construction, the latter functor is weight-exact (see Remark \ref{rk:weight-exact}), symmetric monoidal, and induces a $\otimes$-equivalence of categories between $\mathrm{Kar}(\mathrm{S}\cA(k;R) \otimes_R \kappa(\mathfrak{p}))$ and $\mathrm{S}\cA(k;\kappa(\mathfrak{p}))$. This shows that the assumptions (A2) are also verified.

Let us now prove that the categories $\mathrm{NMS}\cA(k;\kappa(\mathfrak{p}))$ have the $w_{\kappa(\mathfrak{p})}$-Picard property; thanks to Theorem \ref{thm:main2} this implies that $\mathrm{NMS}\cA(k;R)$ has the $w_R$-Picard property. In order to do so, we will make use of Theorem \ref{thm:main1}. Concretely, we need to prove that the categories $\mathrm{S}\cA(k,\kappa(\mathfrak{p}))$ are semi-simple and local. As explained in \cite[Cor.~2.13]{separable}, the additive symmetric monoidal functor $Z(-)\colon \mathrm{Sep}(k;R) \to \mathrm{AM}(k;R)$ corresponds, under the above models, to the forgetful functor
\begin{eqnarray*}
\mathrm{Cov}'(\Gamma;R) \too \mathrm{Cov}(\Gamma;R) && (S,A) \mapsto S\,.
\end{eqnarray*}
This implies, in particular, that the functor $Z(-)$ is faithful and conservative. Since the category $\cA(k,\kappa(\mathfrak{p}))$ is semi-simple and local (see \S\ref{sub:proof22}), we then conclude that the category $\mathrm{S}\cA(k,\kappa(\mathfrak{p}))$ is also semi-simple and local.
\begin{remark}\label{rk:last}
Note that whenever $R$ is an $\bbQ$-algebra, we have the equality
$$ \Hom_{\mathrm{Cov}'(\Gamma;R)}((S_1,A),(S_2,B))= \Hom_{\mathrm{Cov}(\Gamma; R)}(S_1,S_2)\,.$$
This implies that the functor $Z(-)$ induces a $\otimes$-equivalence of categories between $\mathrm{Sep}(k;R)$, resp. $\mathrm{S}\cA(k;R)$, and $\mathrm{AM}(k;R)$, resp. $\cA(k;R)$.
\end{remark}
\section{Proof of Theorem \ref{thm:computation-10}}
Let us denote by $\mathrm{P}(E)$ the smallest additive, Karoubian, full subcategory of $\cD_c(E)$ containing the $E$-module $E$. Since by assumption the ring spectrum $E$ is connective, we have trivial positive Ext-groups:
\begin{eqnarray*}
\Hom_{\cD_c(E)}(E,E[n]) \simeq \pi_{-n}(E) =0 && n >0\,.
\end{eqnarray*}
This implies that the subcategory $\mathrm{P}(E)\subset \cD_c(E)$ is negative. Making use of \cite[Thm.~4.3.2 II and Prop.~5.2.2]{Bondarko-Weight}, we conclude that the category $\cD_c(E)$ carries a bounded weight structure $w$ with heart $\mathrm{P}(E)$.

Let us now show that the category $\cD_c(E)$ has the $w$-Picard property. By construction, $\mathrm{P}(E)$ identifies with the category of finitely generated projective $\pi_0(R)$-modules. Therefore, by taking $R:=\pi_0(E)$, all the assumptions (A1) of Theorem \ref{thm:main2} are verified. In what concerns assumptions (A2), take for $\cT_{k(\mathfrak{p})}$ the category $\cD^b(k(\mathfrak{p}))$, equipped with the canonical bounded weight structure with heart $\mathrm{Vect}(k(\mathfrak{p}))$, and for $\iota_{k(\mathfrak{p})}$ the (composed) base-change functor
$$\cD_c(E) \stackrel{-\wedge_E H\pi_0(E)}{\too} \cD_c(H\pi_0(E))\simeq \cD_c(R) \stackrel{-\otimes_R k(\mathfrak{p})}{\too} \cD^b(k(\mathfrak{p}))\,.$$
By construction, the latter functor is weight-exact (see Remark \ref{rk:weight-exact}), symmetric monoidal, and induces a $\otimes$-equivalence of categories between $\mathrm{Kar}(\mathrm{P}(E) \otimes_R \kappa(\mathfrak{p}))$ and $\mathrm{Vect}(k(\mathfrak{p}))$. Since the categories $\cD^b(k(\mathfrak{p}))$ clearly have the $w_{k(\mathfrak{p})}$-property, we conclude from Theorem \ref{thm:main2} that $\cD_c(E)$ has the $w$-Picard property.

Finally, since the category $\cD_c(E)$ has the $w$-Picard property, we have an isomorphism $\mathrm{Pic}(\cD_c(E))\simeq \mathrm{Pic}(\mathrm{P}(E))\times \bbZ$. The proof follows now from the fact that $\mathrm{Pic}(\mathrm{P}(E))$ is isomorphic to $\mathrm{Pic}(\pi_0(E))$.

\medbreak\noindent\textbf{Acknowledgments:} M.~Bondarko is grateful to Vladimir Sosnilo, Qiaochu Yuan, and  to the users of the Mathoverflow forum for their really helpful comments. G.~Tabuada is grateful to Joseph Ayoub, Andrew Blumberg, and Niranjan Ramachandran for useful conversations. The authors would also like to thank Tom Bachmann for commenting a previous version of this article and for kindly informing us that some related results, concerning the categories $\mathrm{DM}(k;R)$ and $\mathrm{SH}(k)$ and whose proofs use others methods, will appear in his Ph.D. thesis.

\end{document}

\end{proof}